\documentclass[10pt]{article}

\usepackage{mathptmx}
\usepackage{wrapfig}
\usepackage{eucal}
\usepackage{amsmath}
\usepackage{amscd}
\usepackage{amssymb}
\usepackage{amsthm}
\usepackage{xspace}
\usepackage[all,tips]{xy}
\usepackage[dvips]{graphicx}
\usepackage{graphics}
\usepackage{epsfig}
\usepackage{verbatim}
\usepackage{syntonly}
\usepackage{hyperref}
\usepackage{amssymb, color, url}

\newtheorem{introtheorem}{Theorem}
\newtheorem{theorem}{Theorem}[section]
\newtheorem{proposition}[theorem]{Proposition}

\newtheorem{corollary}[theorem]{Corollary}
\newtheorem{lemma}[theorem]{Lemma}

\input xy
\xyoption{all}

\DeclareMathOperator{\SL}{SL}

\DeclareMathOperator{\G}{\Phi}

\DeclareMathOperator{\Z}{\mathbb{Z}}

\def\mf#1{\mathfrak{#1}}

\def\Z{\mathbb{Z}}
\def\N{\mathbb{N}}

\usepackage{fancyhdr}

\fancypagestyle{plain}

\begin{document}


\title{\textbf{Full residual finiteness growths of nilpotent groups}}
\author{Khalid Bou-Rabee\thanks{K.B. supported in part by NSF grant DMS-1405609}\and Daniel Studenmund}
\maketitle


\begin{abstract}
Full residual finiteness growth of a finitely generated group $G$ measures how efficiently word metric $n$-balls of $G$ inject into finite quotients of $G$.
We initiate a study of this growth over the class of nilpotent groups.
When the last term of the lower central series of $G$ has finite index in the center of $G$ we show that the growth is precisely $n^b$, where $b$ is the product of the nilpotency class and dimension of $G$.
In the general case, we give a method for finding an upper bound of the form $n^b$ where $b$ is a natural number determined by what we call a terraced filtration of $G$.
Finally, we characterize nilpotent groups for which the word growth and full residual finiteness growth coincide.
\end{abstract}
\vskip.1in
{\small{\bf keywords:}
\emph{residually finite, nilpotent, full residual finiteness, growth.}}

\section*{Introduction}

Let $G$ be a residually finite group endowed with a word metric given by a finite generating set $X$.
A subset $S \subseteq G$ is \emph{fully detected} by a group $Q$ if there exists a homomorphism $\varphi: G \to Q$ such that $\varphi |_S$ is injective.
For a natural number $n$, set $\G_G^X(n)$ to be the minimal order of a group $Q$ that fully detects the ball of radius $n$ in $G$ (first studied in \cite{BM11}).
The \emph{full residual finiteness growth of $G$ with respect to $X$}
is the growth of the function $\G_{G}^{X}$, that is, its equivalence
class under the
equivalence relation defined by $f\approx g$ if and only if there is a
constant $C$ so that $f(n) \leq Cg(Cn)$ and $g(n)\leq Cf(Cn)$ for all
natural numbers $n$.
The growth of $\G_G^X$ is independent of choice of generating set $X$ (see Lemma \ref{lem:containment}).  Therefore
full residual finiteness growth is an invariant of a finitely
generated group, and can be denoted simply $\G_G$.

This article focuses on finitely generated nilpotent groups. 
While it is known that \emph{word growth} (defined below after Theorem
\ref{MainTheorem}) has precisely polynomial growth over this class of
groups \cite{Bass72}, computing other growth functions
for this class has proved to be a serious task. 
Indeed, even computing the answer for \emph{subgroup growth} \cite{MR1978431} in the two-generated free nilpotent case takes work; see \cite{MR2342452}.
The main difficulty lies in that the structure of $p$-group quotients of a fixed finitely generated nilpotent group can depend heavily on the choice of the prime $p$.
That is, it is difficult to draw global behavior (behavior over all finite quotients) from local behavior (behavior over all finite quotients that are $p$-groups).
Moreover, comparisons between full residual finiteness growth and
 word growth, which is to our knowledge the only nontrivial growth
 function known to have precisely polynomial growth over the class of
 nilpotent groups, do not allow one to immediately draw much information on $\G_G^X$.
In fact, the growth of $\G_G^X$ is often, but not always, strictly larger than the word growth of $G$ (see Theorem~\ref{thm:abelian}).
Obtaining \emph{sharp} control of full residual finiteness growth
over this class requires new understanding of the structure theory of
nilpotent groups.

To present our findings, we begin with some basic examples. In \S\ref{sec:examples}, we show that $\G_G(n) = n^k$ for $G=\Z^k$ and
$\G_G(n) = n^6$ for $G$ equal to the discrete Heisenberg group. The key
property shared by these examples is that the center of $G$ is equal
to the last term of the lower central series of $G$. In
fact, we explicitly compute full residual
finiteness growths for all groups satisfying a slightly weaker condition.
To make this precise we introduce notation: for a nilpotent group $G$
of class $c$ we denote by $\gamma_c(G)$ the last nontrivial term of
its lower central series, by $Z(G)$ its center, and by $\dim(G)$ its
dimension. (See \S\ref{sec:nilpotent} for more explicit definitions.)

\begin{introtheorem} \label{theorem:zl} Let $G$ be a finitely generated
  nilpotent group of class $c$ with $[Z(G):\gamma_c(G)] < \infty$.
  Then
  $$
  \G_G(n) \approx n^{c \dim(G)}.
  $$
\end{introtheorem}

The conclusion of Theorem \ref{theorem:zl} does not generally hold when $[Z(G) :
\gamma_c(G)] = \infty$. This is seen by taking $G$ to be the
direct product of the discrete Heisenberg group with $\Z$, which
satisfies $\G_G(n) \approx n^7$ while $c=2$ and $\dim(G) = 4$.
Groups not satisfying the hypothesis of Theorem \ref{theorem:zl} are
generally more complicated than this example.
For instance, in Proposition \ref{D22example} we provide an example of
a nilpotent group $\Gamma$ of class $c=3$ with $\dim(\Gamma) = 8$ and
$\G_\Gamma(n) \approx n^{22}$ that does not split as a direct product.

For general nilpotent group $G$ of class $c$, we introduce methods to find an upper
bound on the polynomial degree of $\G_G$.
Define a {\em terraced filtration} of $G$ to be a filtration $1 = H_0 \leq
H_1 \leq \dotsb \leq H_{c-1}\leq G$ where each $H_i$ is a maximal normal
subgroup of $G$ satisfying $H_i\cap \gamma_{i+1}(G) = 1$. 
Every terraced filtration of $G$ gives an explicit polynomial upper bound
on growth of $\G_G$.

\begin{introtheorem} \label{MainTheorem}
Let $G$ be a finitely generated nilpotent group of class $c$.
Suppose $1=H_0\leq H_1 \leq \dotsb H_{c-1}\leq G$ is a terraced
filtration of $G$. Then
\[
\G_G(n) \preceq n^{c \dim(G) - \sum_{i=1}^{c-1} \dim(H_i)}.
\]
\end{introtheorem}
\noindent
This upper bound generally depends on the choice of terraced
filtration. See the comments following the proof of Theorem
\ref{MainTheorem} in \S\ref{MainProofSection} for an explicit example
demonstrating this dependence. It would be interesting to determine
whether the lowest upper bound obtained from a terraced filtration
by Theorem \ref{MainTheorem} is optimal.

Results on distortion in nilpotent groups from Osin \cite{MR1872804} and Pittet \cite{Pittet97} play an important role in all of our proofs.

We also compare full residual finiteness growth to word growth.
Recall that the {\em word growth}, $w_G$, of a finitely generated group $G$ is the
growth of the function $w_G^X(n) = \left\lvert B_G^X(n) \right\rvert$,
which is independent of $X$.
Gromov \cite{MR623534} has characterized nilpotent groups in the class
of finitely generated groups as those for which $w_G$ is polynomial. 
By applying this theorem, it is shown in \cite{BM11} (see Theorem 1.3 there) 
that full residual finiteness growth enjoys the same conclusion.
In spite of this similarity, these two growths rarely coincide.
Our final result characterizes nilpotent groups for which full residual
finiteness growth equals word growth.

\begin{introtheorem} \label{thm:abelian}
Let $G$ be a finitely generated nilpotent group.
Then $\G_G \approx w_G$ if and only if $G$ is virtually abelian.
\end{introtheorem}
\noindent
In \S\ref{sec:fullresidualfinitenessgrowth} we provide a geometric
interpretation of $\G_G^X$. From this point of view, Theorem
\ref{thm:abelian} implies that virtually abelian groups are characterized in the
class of finitely generated nilpotent groups solely in terms of the
asymptotic data of the Cayley graph.
A non-normal version of full residual finiteness growth,
  the \emph{systolic growth}, is studied in \cite{YC14}.
There it is shown that systolic growth matches word growth if and only if the group is \emph{Carnot}.

This paper is organized as follows: In \S\ref{BackgroundSection} we
present basic results on nilpotent groups and full residual finiteness
growth, including important lemmas on
word metric distortion of central subgroups of nilpotent groups
following from work of Osin \cite{MR1872804} and Pittet \cite{Pittet97}. In
\S\ref{sec:abeliangroups} we prove Theorem \ref{thm:abelian}. In
\S\ref{sec:heisenberg} we compute the full residual
finiteness growth of the Heisenberg group and prove
Theorem \ref{theorem:zl}. 
In \S\ref{MainProofSection} we give an illustrative example showing that the
conclusion of Theorem \ref{theorem:zl} does not hold in general, and
prove Theorem \ref{MainTheorem}.

We finish the introduction with a bit of history.
The concept of full residual finiteness growth was first studied by
Ben McReynolds and K.B. in \cite{BM11}. The full
  residual finiteness growth of the discrete Heisenberg group is
  presented in \cite{YC14}.
Compare full residual finiteness growth to the concept of \emph{residual finiteness growth}, which measures
how well individual elements are detected by finite quotients,
appearing in \cite{B09},
\cite{MR2583614}, \cite{BM13}, \cite{KM12}, \cite{R12}, \cite{BK12}, \cite{KMS13}.
Also compare this with Sarah Black's \emph{growth function} defined and studied in \cite{MR1659911}. 
Full residual finiteness growth measures how efficiently the word growth function 
can be recovered from Black's growth function. See remarks in \cite{MR1659911} on p.\ 406 before \S 2 for further discussion.

\paragraph*{Acknowledgements}

The authors are grateful to Benson Farb for suggesting this pursuit.
The authors acknowledge useful conversations with Moon Duchin, Michael
Larsen, Ben McReynolds, Christopher Mooney, and Denis Osin. The authors are
further grateful to Benson Farb and Ben McReynolds for comments on
drafts of this paper.
K.B. gratefully acknowledges support from the AMS-Simons Travel Grant Program.
The authors are very grateful to the excellent referee for comments
and corrections that greatly improved the paper and for suggesting
Proposition \ref{prop:Treduction}.

\section{Some background and preliminary results} \label{BackgroundSection}

\subsection{Full residual finiteness growth} \label{sec:fullresidualfinitenessgrowth}

In this subsection we give a geometric interpretation of full residual
finiteness growth for finitely presented groups.

Write $f \preceq g$ to mean there exists $C$ such that $f(n) \leq C g(Cn)$.
We write $f \approx g$ if $f \preceq g$ and $g \preceq f$.
Recall that the \emph{growth} of a function $f$ is the equivalence class of $f$ with respect to $\approx$.

We first prove a lemma that implies that the growth of the function
$\G_G^X$ defined in the introduction is independent of generating set
$X$:

\begin{lemma} \label{lem:containment}
Let $G$ be finitely generated with finitely generated subgroup $H \leq G$.
Fix finite generating sets $X$ and $Y$ for $G$ and $H$.
Then $\G^Y_H \preceq \G_G^X$.
\end{lemma}

\begin{proof}
Since $H \leq G$, there exists $C > 0$ such that any element in $Y$ can be written in terms of at most $C$ elements in $X$.
Thus, $B_H(n) \subseteq B_G(Cn)$ for any $n > 1$.
Because any homomorphism from $G$ restricts to a homomorphism from
$H$, this gives
$$
\G_G^X(Cn) \geq \G_H^Y(n),
$$
as desired.
\end{proof}

Lemma \ref{lem:containment} in particular implies that if $X$ and $Y$
are two finite generating sets of a group $G$, then $\G_G^X \approx
\G_G^Y$. Let $\G_G$ denote the equivalence class of $\G_G^X$ with
respect to $\approx$ for any finite generating set $X$ of $G$.

We now provide a geometric interpretation of $\G_G$ in the case that
$G$ is a finitely presented group.
Let $G$ be a residually finite group with Cayley graph $\Gamma$ with
respect to a finite generating set $S$.
Each edge of $\Gamma$ is labeled by the corresponding generator.
For a subset $X \subseteq \Gamma$, we set $\partial X$ to be the collection of edges and vertices of $X$ each of which has closure not contained in the interior of $X$.
Let $\{ A_k \}$ be an increasing sequence of finite connected subsets of $\Gamma$ with 
$$
A_{k+1} = \partial A_{k+1}  \sqcup A_k.
$$
Then the sequence of subsets, $\{A_k \}$, is called a \emph{growing
  sequence}.  Let $B_{G}^S(n)$ denote the closed ball of radius $n$ in
the Cayley graph of $G$ with respect to the word metric induced by
$S$. We will omit the $S$ from the notation when the generating set is
understood and there is no chance for confusion. The prototypical
example of a growing sequence is the sequence that assigns to each
positive integer $k$ the metric ball $B_{G}^S(k)$ in the Cayley graph
of $G$ with respect to $S$.

The \emph{geometric full residual finiteness growth of $\Gamma$ with respect to
  $\{ A_k \}$} is the growth of the function, $\G_{\Gamma}^{\{ A_k \}}: \N \to \N$, given by 
\begin{eqnarray*}
n \mapsto \min \{ |Q| : \text{ $Q$ is a group with $A_n$ isometrically} \\
\text{embedding in one of its Cayley graphs}\}.
\end{eqnarray*}

Our first lemma demonstrates that the growth of $\G_G^{\{A_k \}}$ does
not depend on the growing sequence.

\begin{lemma} \label{lem:indepofgrowingset}
Let $\{ X_k \}$ and $\{ Y_k \}$ be two growing sequences for a finitely generated group $G$.
Then $\G_{G}^{\{X_{k} \}} \approx \G_{G}^{\{ Y_{k} \}}$.
\end{lemma}

\begin{proof}
We first assume that $\{ X_k \}$ and $\{ Y_k \}$ are growing sequences from the same Cayley graph realization of $G$.
Then there exists $K \in \N$ such that
$$
Y_1 \subseteq X_K \text{ and } X_1 \subseteq Y_K.
$$
Hence, $C_G^{\{Y_k\}} (n) \leq C_G^{\{X_{k}\}}(K+i) \text{ and } C_G^{\{X_k\}} (n) \leq C_G^{\{Y_{k}\}}(K+i).$
Thus, we can assume that $\{ X_k \}$ and $\{ Y_k \}$ are the word metric $k$-balls of $G$ with respect to two different generating sets.
It is straightforward to see that there exists $C > 0$ such that $Y_{n} \subseteq X_{Cn} \subseteq Y_{C^2n}$ for every natural number $n$.
Hence,
$$
\G_G^{\{Y_k\}} (n) \leq \G_G^{\{X_k\}} (Cn) \leq \G_G^{\{Y_k\}} (C^2 n),
$$ as desired.
\end{proof}

Next we show that the notions of full residual finiteness growth,
given in the introduction, and geometric full residual
finiteness growth, given in this section, agree in the case that the group
$G$ is finitely presented.
It would be interesting to determine if this equivalence holds for all finitely generated groups.

\begin{lemma} \label{lem:cayleyvsgirth}
Let $G$ be a finitely presented group.
For any generating set $X$ and growing sequence $\{ A_k \}$ we have
\[\G^{\{ A_k \}}_G \approx \G^{X}_G.\]
\end{lemma}

\begin{proof}
Let $X$ be a finite generating set for $G$ and let $R$ be the set of finite relations.
It is clear that $\G^{\{ A_k \}}_G \preceq \G^{X}_G$. We show the reverse inequality.
We can, by Lemma \ref{lem:indepofgrowingset}, suppose that the growing set $\{ A_k \}$ is simply the sequence $ \{ B_G^X(k) \}$.
It suffices, then, to show that there exists $N \in \N$ such that for any $n > N$ and any finite group, $Q$, with $B_G(n)$ isometrically embedding in a Cayley graph realization of $Q$, there exists a homomorphism
$\phi: G \to Q$ with $\phi |_{B_G(n)}$ being injective.
Select $N$ to be the maximal word length of any element in $R$.
Then since $B_n$ isometrically embeds in a Cayley graph of $Q$, we see that there exists a generating set for $Q$ such that each relator $R$ is satisfied by this generating set.
This finishes the proof.
\end{proof}

The next lemma controls some of the full residual finiteness growth of a direct product of groups.

\begin{lemma} \label{lem:directproducts}
Let $G$ and $H$ be finitely generated groups.
Then 
$$
\G_{G \times H} \preceq \G_G \cdot \G_H.
$$
\end{lemma}

\begin{proof}
  Fix generating sets $X$ and $Y$ for $G$ and $H$. Then $(X \times \{
  1 \}) \cup ( \{1\} \times Y )$ is a finite generating set for $G \times H$.
  Note that
  $$
  B_{G\times H} (n)  \subseteq (B_G (n) \times \{1\}) (\{1\} \times B_H(n)).
  $$
  Thus, if $Q_1$ is a quotient that fully detects $B_G(n)$ and $Q_2$ a quotient that fully detects $B_H(n)$, then $Q_1 \times Q_2$ fully detects $B_{G \times H}(n)$.
  We see then that $\G_{G \times H} \preceq \G_G \G_H$, as desired.
\end{proof}
\noindent
Can the conclusion of Lemma \ref{lem:directproducts} be improved to $\G_{G\times H} \approx \G_G \G_H$?
This can possibly be false: it is not even true that if $\varphi : G \to H$ is a surjective homomorphism, then $\G_G(n) \succeq \G_H(n)$. Consider a free group mapping onto one of Kharlampovich-Sapir's solvable and finitely presented groups of arbitrarily large residual finiteness growth \cite{KMS13}.

Full residual finiteness growth is well-behaved under taking the
quotient by a finite normal subgroup:

\begin{proposition} \label{prop:Treduction}
	Let $G$ be a finitely generated residually finite group.
	Let $T$ be a finite normal subgroup of $G$.
	Then $\Phi_G \approx \Phi_{G/T}$.
\end{proposition}

\begin{proof}
	Fix a generating set $X$ for $G$, and let $Y$ be the image of $X$ under the quotient map $G \to G/T$. Let $K$ be the largest length, with respect to $X$, of an element in $T$.
	We first claim $\Phi_G^X(K+n) \geq \Phi_{G/T}^Y(n)$.
	Let $\phi: G \to Q$ be a finite quotient of minimal cardinality that fully detects $B_G^X(K+n)$. That is $|Q| = \Phi_G^X(K+n)$.
	Define $\psi : G/T \to Q/\phi(T)$ by $gT \mapsto \phi(g) \phi(T)$.
	Let $g \in B_{G/T}^Y(n) \cap \ker \psi$. 
	By construction, we may lift $g$ to an element $\tilde g \in G$ such that $\tilde g \in B_{G}^X(n)$ and $\phi(\tilde g) \in \phi(T)$.
	That is, there exists $t \in T$, such that $\phi(g) = \phi(t)$, which gives
	$$
		\phi(\tilde g t^{-1}) = 1.
	$$
	If $\tilde g t^{-1} \neq 1$, then this contradicts that $\phi$ fully detects $B_G(K+n)$. Hence, $\tilde g = t$, and so $\ker \psi \cap B_{G/T}^Y(K+n)$ is trivial. It follows that $\psi$ fully detects $B_{G/T}^Y(n)$, and so $\Phi_G^X(K+n) \geq \Phi_{G/T}^Y(n)$, as claimed.

	Since $G$ is residually finite and $T$ is finite, there exists a normal subgroup, $H$, such that $T\cap H = 1$.
	To finish, we claim that $\Phi_{G}^X(n) \leq [G:H]\Phi_{G/T}^Y(n)$.
	Let $\psi : G/T \to Q$ be a quotient that fully detects $B_{G/T}^Y(n)$, with $|Q| = \Phi_{G/T}^Y(n)$.
	Let $\phi : G \to Q$ be the natural map $G \to G/T \to Q$.
	Set $N = \ker \phi \cap H$.
	Clearly, $[G: N] \leq [G: \ker \phi] [G:H] = |Q| [G:H]$.
	Moreover, if $g \in B_G^X(n) \cap N$, then $g \notin T$.
	Hence, by the construction of $Y$, we have that $\phi(g) \neq 1$.
	It follows that $G/N$ fully detects $B_G^X(n)$, and so
	$\Phi_{G}^X(n) \leq [G:H]\Phi_{G/T}^Y(n)$, as desired.
\end{proof}

 We finish the section with a lemma that, in some restrictive cases, allows us to pass to finite-index subgroups.

\begin{lemma} \label{lem:finiteindex}
Let $G$ and $H$ be finitely generated nilpotent groups with $H$ normal
subgroup in $G$ of finite index.
If every normal subgroup of $H$ is normal in $G$, then $\G_G \approx \G_H$.
\end{lemma}

\begin{proof}
By Lemma \ref{lem:containment}, it suffices to show that $\G_G \preceq \G_H$.
Fix generating sets for $G$ and $H$ so that $B_H(n) \subseteq B_G(n)$
for all $n>0$.
Because $H$ is of finite index in $G$ and thus quasi-isometric to $G$, there exists $C > 0$ such that $H \cap B_G(2n) \subseteq B_H(Cn)$.
Let $H/K$ be a quotient of $H$ that fully detects $B_H(Cn)$.
By our assumption, $K$ is normal in $G$ so $G/K$ is well-defined.
Then any element in $B_G(2n)$ not in $H$ is mapped nontrivially onto $G/K$.
And since $H \cap B_G(2n) \subseteq B_H(Cn)$, it follows that $B_G(2n)$ is mapped nontrivially onto $G/K$.
Thus, $B_G(n)$ is fully detected by $G/K$, and so we are done.
\end{proof}

\subsection{Nilpotent groups} \label{sec:nilpotent}

In this subsection we fix basic notation and present several lemmas that play important roles in our proofs.
Let $G$ be a group. The \emph{lower central series} $\gamma_{k} (G) $ of $G $ is the sequence of subgroups defined by $\gamma_{1}(G) = G$ and 
$$\gamma_{k}(G) =[\gamma_{k-1}(G),G].$$ 
For any group $H$, let $Z(H)$ denote the center of $H$.
The \emph{upper central series} $\zeta_k (G) $ of $G $ is
given by $\zeta_{0}(G)=\{e\}$ and the formula 
$$
\zeta_{k}(G) / \zeta_{k - 1} (G) =
Z(G/\zeta_{k-1}(G)).
$$ 
The group $G$ is said to be \emph{nilpotent} if $\gamma_{k}(G) = 1$
for some natural number $k$.  Equivalently, $G$ is nilpotent if and
only if it is an element of its upper central series. Moreover, $G$ is
said to be \emph{nilpotent of class $c$} if $\gamma_c(G) \neq 1$ and
$\gamma_{c+1}(G) = 1$.

If $G$ is a finitely generated nilpotent group, then the successive quotients of the upper central series of $G$ are abelian groups of finite-rank.
Thus, the upper central series has a refinement
$$G =  G_{1} \ge G_{2} \ge \ldots G_{n+1} = 1,$$
such that $G_{i} / G_{i+1} $ is cyclic for all $i = 1, \ldots, n$.  
The number of infinite cyclic factors in this series does not depend on the series and is called the \emph{dimension} of $G$, denoted by $\dim(G)$ \cite[p. 16, Exercise 8]{MR713786}.
Let this series be chosen so that $n$ is minimal.
An $n$-tuple of elements $(g_{1}, g_{2}, \ldots , g_{n})\in G^ n$
is a \emph{basis} for $G $ if $g_{i}\in G_{i}$ and
$G_i/G_{i-1}=\left <g_{i} G_{i - 1} \right >$ for
each $i = 1,\ldots, n $.
In the case when $G_{i} / G_{i+1} $ is infinite for all $i = 1, \ldots, n$ we call the $n$-tuple a \emph{Malcev basis} for $G$.

The set of torsion elements $T$ in a finitely generated nilpotent
group $G$ is a finite normal subgroup, and the quotient $G/T$ is a
torsion-free nilpotent group \cite[p. 13, Corollary 10]{MR713786}. A
corollary of Proposition \ref{prop:Treduction} is that $G/T$ has the
same full residual finiteness growth as $G$.

\begin{corollary}\label{cor:tfreduction}
  If $G$ is a finitely generated nilpotent group and $T$ is the
  subgroup of torsion elements then $\Phi_G \approx \Phi_{G/T}$.
\end{corollary}

We recall a folklore result, used in the proof of the following
lemmas.

\begin{lemma} \label{lem:commutatorproduct}
  Suppose $G$ is a finitely generated nilpotent group of class $c$. The assignment $(x, y)
  \mapsto [x,y]$ defines a homomorphism
\[
\left( \zeta_k(G)/\zeta_{k-1}(G) \right) \times \left( \zeta_\ell(G) /
  \zeta_{\ell-1}(G) \right) \to \zeta_{k+\ell-c-1}(G) / \zeta_{k+\ell-c-2}(G).
\]
\end{lemma}
\begin{proof}
  This follows immediately from \cite[Theorem 2.1]{MR2366181}, noting
  that the upper central series is a central filtration of $G$ when
  indexed so that the $i^{th}$ term of the filtration is
  $\zeta_{c+1-i}(G)$.
\end{proof}

If $G$ is a group generated by a finite set $X$, for $g\in G$ we use
$\| g \|_X$ to denote the word length of $g$ with respect to $X$. 
Let $G$ be a finitely generated nilpotent group.
The following lemma is a consequence of well-known distortion estimates.

\begin{lemma} \label{lem:distortion}
Let $G$ be a nilpotent group of class $c$ generated by a finite set $X$.
Fix a positive integer $i$ and a generating set $X_i$ for $Z(G) \cap \gamma_i(G)$.
Then there exists $C > 1$ such that for all $g \in Z(G) \cap \gamma_i(G),$
\begin{equation} \label{eq:distortion}
\|g \|_{X} \leq C \| g \|_{X_i}^{1/i}. 
\end{equation}
\end{lemma}

\begin{proof}
We first assume $G$ is torsion-free. 
First consider the case that $g = x^m$ for some $x\in X_i$ and $m\in
\mathbb{Z}\setminus\{0\}$. 
Assume without loss of generality that $X_i$ is a basis for the free
abelian group $Z(G) \cap \gamma_i(G)$, so that $\| g \|_{X_i} = \left\lvert
  m \right\rvert$.
Embed $G$ as a cocompact lattice in a simply-connected nilpotent Lie group $N$, which
identifies $Z(G) \cap \gamma_i(G)$ with a lattice in a
simply-connected central subgroup $Z\leq N$. Fix any left-invariant
Riemannian metric on $N$, which gives a norm $\| \cdot \|_\mf{n}$ on
$\mf{n}$, the Lie algebra of $N$. Consider the path $\gamma : [0,
\left\lvert m \right\rvert ] \to Z$ defined so that $\gamma(
\left\lvert m \right\rvert ) = g$ and
$\gamma(t) = \exp( t z )$ for some $z\in \mf{n}$. Note that $\exp(z) =
x$ if $m>0$ and $\exp(z) = x^{-1}$ if $m<0$. In particular, $z$ does \emph{not} depend on $m$.
By \cite[Prop 4.1(1)]{Pittet97}, the length of $\gamma$ is 
$\| z \|_\mf{n} \| g\|_{X_i}$. 
Then applying
\cite[Prop 4.1(2)]{Pittet97} to the curve $\gamma$, there is a
constant $C>0$ depending on $z$ so that
\begin{equation} \label{pitteteqn}
d_N(e, g) \leq C \|g\|_{X_i}^{1/i}.
\end{equation}
The quantity $d_N(e,g)$ is uniformly comparable to $\| g \|_X$, so
this proves the desired inequality for $g$ of the form $x^m$.

Now for any $g\in Z(G)\cap \gamma_i(G)$, write $g = \prod_{j=1}^k
x_j^{m_j}$ where $X_i = \{x_1,\dotsc, x_k\}$. Let $C$ be the largest
constant appearing in equation \ref{pitteteqn} as $x$ ranges over
$x_1,\dotsc, x_k$. Then there is some $D > 0$ so that
\begin{align*}
  \| g \|_X & \leq \sum_{j=1}^k \| x_j^{m_j} \|_X \\
  & \leq C \sum_{j=1}^k \| x_j^{m_j} \|_{X_i}^{1/i} \\
  & \leq C \sum_{j=1}^k |m_j|^{1/i} \\
  & \leq C k \left( \sum_{j=1}^k | m_j | \right)^{1/i} \\
  & \leq CkD \| g \|_{X_i}^{1/i}.
\end{align*}
The last step follows because $Z(G) \cap \gamma_i(G)$ is abelian. 
The penultimate step follows from the general fact that $(m_1^{1/i} +
\dotsb + m_k^{1/i})^i \leq k^i (m_1 + \dotsb + m_k)$ when $m_j \geq 1$
for all $j$. This completes the proof in the case that $G$ is torsion-free.

Now suppose $G$ is an arbitrary finitely generated nilpotent group.
There is a torsion-free normal subgroup $H$ of finite index in $G$.  Fix a
generating set $Y$ for $H$.  The map $i : H \to G$ is a quasi-isometry
because $H$ is finite index in $G$. In fact, because distinct points
in each of $G$ and $H$ are distance at least 1 and $i$ is injective,
it is easy to check that $i$ is bi-Lipschitz.  This means that there is
some $C\geq 1$ so that:
\begin{enumerate}
	\item For $g,h \in H$, 
	$$
		\frac{1}{C} \| g h^{-1} \|_Y \leq \| g h^{-1} \|_X \leq C \| g h^{-1} \|_Y.
	$$
	\item For every element $g \in G$, there exists $h \in H$ such that
	$$
		\| h g^{-1} \|_X \leq C.
		$$
\end{enumerate}

Fix generating sets $X_i$ for $Z(G) \cap \gamma_i(G)$
and $Y_i$ for $Z(H) \cap \gamma_i(H)$. 
We claim that $Z(H) \leq Z(G)$.
Indeed, if not then there exists $h \in Z(H)$, an integer $r \geq 1$,
and elements $x_1, \ldots, x_r \in G$ such that $h \in \zeta_{r+1}(G)
\setminus \zeta_r(G)$ and
$$
[h, x_1, \ldots, x_r] \in Z(G) \setminus \{1\}.
$$
Since $H$ has finite index in $G$ there exists $n \in \N$ such that $x_1^n \in H$.
By Lemma \ref{lem:commutatorproduct} we have 
\[
[h, x_1^n, \ldots, x_r] = [h, x_1, \ldots, x_r]^n.
\]
Since $H$ is normal we have $[h, x_1, \ldots, x_r] \in H$.
This implies $[h, x_1^n, \ldots, x_r] \neq 1$ because $H$ is torsion-free.
Therefore $[h,x_1^n]$ cannot be trivial, which contradicts the fact that $h \in Z(H)$.
By the aforementioned claim, $Z(G) \cap \gamma_i(G)$ contains $Z(H) \cap \gamma_i(H)$ as a subgroup.
In fact, it is not hard to show that $Z(H) \leq Z(G)$ and $\gamma_i(H)
\leq \gamma_i(G)$ are, in both cases, subgroups of finite index.
Hence, the inclusion
$i_2 : Z(H) \cap \gamma_i(H) \to Z(G) \cap \gamma_i(G)$ is a
bi-Lipschitz quasi-isometry with constant $D\geq 1$.

Now select $C' > 1$ such that inequality \ref{eq:distortion} holds for all
$g \in G$ that are finite order (again, there are only finitely many of them).
Next, let $g$ be an infinite order element in $G$ with 
$$g \in Z(G) \cap \gamma_i(G).$$
We can suppose, without loss of generality, that $X$ contains $X_i$.
Then since $i_2$ is a $D$-quasi-isometry, there exists $h \in Z(H) \cap \gamma_i(H)$ such that 
$$ \| h g^{-1} \|_{X} \leq  \| h g^{-1} \|_{X_i} \leq D.$$
Since $H$ is torsion-free, by enlarging $C$ if necessary we have
$$
\| h \|_Y \leq C \| h \|_{Y_i}^{1/i}.
$$
Thus,
\begin{equation} \label{QI1}
\| g \|_X = \| g h^{-1} h \|_X \leq \| g h^{-1} \|_X + \| h \|_X
\leq C + \| h \|_X.
\end{equation}
And, further,
\begin{equation} \label{QI2}
\| h \|_X \leq C \| h \|_Y \leq C^2 \| h \|_{Y_i}^{1/i}.
\end{equation}
To finish,
\begin{equation} \label{QI3}
\| h \|_{Y_i} \leq D \| h \|_{X_i} =  D \| g g^{-1} h \|_{X_i} \leq D(\| g \|_{X_i} + D ).
\end{equation}
The desired inequality follows from equations \ref{QI1}--\ref{QI3}, as
all additive constants can be absorbed into the multiplicative
constants.
\end{proof}

Next,  we show a technical lemma that will be important in our main proofs:

\begin{lemma} \label{lem:technical}
Let $G$ be a nilpotent group of class $c$ generated by a finite set
$X$. Fix a number $0<i\leq c$, and fix generating sets $Y_0$ for $\zeta_i(G)$ and $Y$ for $Z(G)$.
There exists $C_i > 0$ such that for any $g \in \zeta_i(G) \setminus \zeta_{i-1}(G)$, there exists $x_1, \ldots, x_{i-1} \in X$ such that for any $\gamma \in \zeta_{i-1}(G)$,
$$
0 < \|[g \gamma, x_1, \ldots, x_{i-1}] \|_Y \leq C_i \|g \|_{Y_0}.
$$
In fact, there is some $F_i > 0$ so that
$$
0 < \|[g \gamma, x_1, \ldots, x_{i-1}] \|_{X} \leq F_i \|g \|_{Y_0}^{1/t}, 
$$
where $t$ is the minimal $k$ satisfying $[g \gamma, x_1, \ldots, x_{i-1}] \notin \gamma_{k+1}(G)$.
\end{lemma}

\begin{proof}
Let $g \in \zeta_i(G) \setminus \zeta_{i-1}(G)$ be given.
Since $G$ is nilpotent, there exists $x_1, \ldots, x_{i-1} \in X$ so that
$$
[g, x_1, \ldots, x_{i-1}] \in Z(G) \setminus \{1 \}.
$$
Note that for any $x \in \zeta_i(G)$ we have that 
$$
[x, x_1, \ldots, x_{i-1}] \in Z(G).
$$
Write $g = \prod_{i=1}^n g_i$ where $g_i \in Y_0$ and $n$ is the  word
length of $g$ with respect to $Y_0$.
Applying Lemma~\ref{lem:commutatorproduct} repeatedly gives
\begin{eqnarray*}
[g, x_1, \ldots, x_{i-1}] &=& [g_1 g_2 \cdots g_n, x_1, \ldots, x_{i-1}] \\
&=& [g_1, x_1, \ldots, x_{i-1}][g_2, x_1, \cdots x_{i-1}] \cdots [g_n, x_1, \ldots, x_{i-1}].
\end{eqnarray*}
Set $Y'$ to be $Y$ union the set of all elements of the form $[\beta, \alpha_1, \alpha_2, \ldots, \alpha_{i-1}]$ where $\beta \in Y_0$ and $\alpha_i \in X$.
Notice that $Y'$ does not depend on $g$. Further, by our above computation, we have
$$
\| [g, x_1, \ldots, x_{i-1}] \|_{Y'} \leq n.
$$
Because $Y'$ is finite, $(Z(G), d_Y)$ is bi-Lipschitz equivalent to
$(Z(G), d_{Y'})$.
This gives $C_i > 0$, depending only on $Y'$, such that
\begin{equation} \label{assertionone}
0 < \| [g, x_1, \ldots, x_{i-1} ] \|_Y \leq C_i \|[g, x_1, \ldots, x_{i-1}] \|_{Y'}
\leq  C_i n = C_i \| g \|_{Y_0}.
\end{equation}
Let $\gamma \in \zeta_{i-1}(G)$ be arbitrary.
Then as $[g, x_1, \ldots, x_{i-1}]$ and $[\gamma, x_1, \ldots, x_{i-1}]$ are central,
$$
[g, x_1, \ldots, x_{i-1} ] = [g \gamma, x_1, \ldots, x_{i-1} ],
$$
so the proof of the first assertion is complete. 

Fix generating sets $X_j$ for $\gamma_j(G) \cap Z(G)$ for each $1\leq
j \leq c$. These sets can be chosen independently of $g$ and $i$.
By Lemma \ref{lem:distortion}, for each $j$ we have that there exists $D_j > 1$ such that for all $w \in \gamma_j(G) \cap Z(G)$
$$
\|w \|_X \leq D_j \| w \|^{1/j}_{X_j}.
$$
Set $D$ to be the maximal such $D_j$. Notice that $D$ only depends on $X$ and $G$.
Since $\gamma_t(G) \cap Z(G)$ is a subset of the abelian group, $Z(G)$, we have that there exists $E > 1$, depending only on $Y$ and the selection of $X_j$, such that
$$
\|[g \gamma, x_1, \ldots, x_{i-1}] \|_{X_t} \leq E \| [g \gamma, x_1, \ldots, x_{i-1}]  \|_{Y} \leq E^2  \| [g \gamma, x_1, \ldots, x_{i-1}] \|_{X_t}.
$$
Combining these inequalities with Inequality \ref{assertionone} gives
\begin{eqnarray*}
0 < \|[g \gamma, x_1, \ldots, x_{i-1}] \|_{X}  &\leq& D \| [g \gamma, x_1, \ldots, x_{i-1} ] \|_{X_t}^{1/t} \\
&\leq&  E D \| [g \gamma, x_1, \ldots, x_{i-1} ] \|_{Y}^{1/t} \\
&\leq&  E D C_i \| g \|_{Y_0}^{1/t} = F_i \| g \|_{Y_0}^{1/t},
\end{eqnarray*}
for some constant $F_i$ that depends only on $i$ and our choice of generating sets, as desired.
\end{proof}

For any $g \in G$ of infinite order, the \emph{weight $\nu_G(g)$ of $g$ in the group $G$} is the maximal $k$ such that $\left< g \right> \cap \gamma_k(G) \neq \{ 1 \}$.
If $G$ is a group and $m$ a natural number, let $G^m$ denote the
normal subgroup of $G$ generated by all $m^{th}$ powers of elements of
$G$. When $G$ is nilpotent we have $[G : G^m] < \infty$ for any $m$
(see, for instance, \cite[p.\ 20, Lemma 4.2]{MR0283083}).
We need the following technical result for Lemma \ref{lem:torsionlengths}.

\begin{lemma}
	\label{lem:distortion1}
	Let $G$ be a nilpotent group generated by a finite set $X$.
Fix a positive integer $i$.
Then there exists a constant $C > 1$ such that for all $m \in \N$ and all $g \in (Z(G) \cap \gamma_i(G))^m$ with $\nu_G(g) = i$, we have 
$$
m \leq C \| g \|^i_X.
$$
\end{lemma}

\begin{proof}
Select $Y = \{ x_1, \ldots, x_r \}$ so that the image of $Y$ under
$$\pi : (Z(G) \cap \gamma_i(G)) \to (Z(G) \cap \gamma_i(G)) / \gamma_{i+1}(G)$$ 
generates a free abelian group of rank $r$, where $r$ is the rank of 
$(Z(G) \cap \gamma_i(G)) / \gamma_{i+1}(G)$.
Select $N$ to be the order of $\pi(Z(G) \cap \gamma_i(G))/\pi(\left<Y\right>)$.

Let $f$ be the projection $G \to G/\gamma_{i+1}(G)$.
We apply \cite[Theorem 2.2]{MR1872804} to the torsion-free subgroup $\Pi = f(\left< x_1, \ldots,
  x_r \right>) \leq G/\gamma_{i+1}(G)$ to get $D > 1$ such that
$$
\sup\limits_{h \in \Pi \cap B_{f(X)}(n)} \| h \|_{f(Y)} \leq D n^{i}.
$$
Thus, if $\| h \|_{f(X)} = n$, then $\| h \|_{f(Y)} \leq D n^i = D (\| h \|_{f(X)})^i$.
That is, 
\begin{equation} \label{eqn:compareh}
\| h \|_{f(Y)} \leq D \| h \|_{f(X)}^i.
\end{equation}

Now suppose $g$ is an element of $(Z(G) \cap \gamma_i(G))^m$ with $\nu_G(g) = i$.
Since $\nu_G(g) = i$, we have that $\pi(g)$ is infinite order.
Thus, we can write $g^N = h \gamma$ where $h \in \left< Y \right>$ and $\pi(\gamma)$ is trivial.
Thus, $\gamma \in \gamma_{i+1}(G)$.
The map $f|_{ \left< Y \right> }$ is an injection, thus
$$
\| f(h) \|_{f(Y)} = \| h \|_{Y}.
$$
Finally, by the fact that $g^N \equiv h \mod \gamma_{i+1}(G)$ and Inequality (\ref{eqn:compareh}), we have
\begin{eqnarray*}
N \| g \|_X &\geq& \| g^N \|_X \geq \| f(g^N) \|_{f(X)} \\
&=& \| f(h) \|_{f(X)} \geq D^{1/i} \| f(h) \|_{f(Y)}^{1/i}.
\end{eqnarray*}
Notice that for any abelian group, $A$, and $\ell \in \N$ we have
$$A^{\ell} = \left< \{ x^{\ell} : x \in A \} \right>.$$
Using additive notation, this becomes
$$
A^{\ell} := \left< \{ \ell x : x \in A \} \right>
= \ell \{ x : x \in A \} = \ell A.
$$
So we have
$$mN A = m ( N A) = m \{ n x : x \in A \}.$$
To apply this, note that $A = (Z(G) \cap \gamma_i(G))$ is an abelian group, as it is contained in the center.
For any element $y \in N A$, by the definition of $N$, we have
$f(y)$ is an element of $\Pi$.
Thus, $f(m y) = m f(y)$ is an element of $\Pi^m$, and so it follows that
$$
f(A^{mN}) \leq \Pi^m.
$$
In particular, $g^N \in (Z(G) \cap \gamma_i(G))^{Nm}$, so we have
$$
f(g^N) \in \Pi^m.
$$
Since $m f(Y)$ is a free basis for $\Pi^m$, $f(Y)$ is a free basis for $\Pi$, and $f(h) = f(g^N)$, we conclude that
$$
\| f(h) \|_{f(Y)} = \| f(g^N) \|_{f(Y)} \geq m,
$$
so we are done.
\end{proof}

With the previous lemmas in hand, we finish with a proof that gives some control on the word lengths of elements in $G^m$.

\begin{lemma} \label{lem:torsionlengths}
Let $\tilde G$ be a finitely generated nilpotent group of nilpotence $c$.
There exists $f \in \N$ such that $G = \tilde G^f$ is a torsion-free characteristic subgroup of $\tilde G$ of finite index.
Let $g \in G$, $X$, and $t \in \N$ be as in Lemma \ref{lem:technical}.
Then there exists $C > 1$, $M \in \N$, depending only on $G$, such that if $g \in G^{Mm}$, we have that
$$
\| g \|_X \geq C m^{1/t}.
$$
\end{lemma}

\begin{proof} 
Set $\tau(\tilde G)$ to be the set of all elements of finite order in $\tilde G$.
By \cite[p.\ 13, Chapter 1, Corollary 10]{MR713786}, this is a finite characteristic subgroup of $\tilde G$.
Since $G$ is residually finite and $\tau(H)$ is finite, there exists a finite $Q$ that fully detects $\tau(\tilde G)$. 
Set $f$ to be the exponent of $Q$ and set $G$ to be the characteristic
finite-index subgroup $\tilde G^{f}$ \cite[p.\ 20, Lemma 4.2]{MR0283083},
Then the map $\tilde G \to Q$ factors through $\tilde G/G$, and thus
$\tau(\tilde G)$ is fully detected by $\tilde G/G$.
Since $\tau(\tilde G)$ contains all the torsion elements in $\tilde G$, it follows that $G$ is torsion-free.

We will show by induction on $d$ that for all $n > c$,
\begin{equation} \label{eqn:claim}
(\zeta_{d}(G))^{(d)! \cdots 2! n} \cap Z(G) \leq Z(G)^{n}.
\end{equation}
The base case $\zeta_1(G) = Z(G)$ is immediate.
For the inductive step, set $M = (d)! (d-1)! \cdots 2!$ and let $H =
\zeta_{d}(G) \leq G$.
Let $h \in H^{Mn} \cap Z(G)$.
Since $h$ is in $H^{Mn}$ we can write
$$
h = g_1^{Mn} g_2^{Mn} \cdots g_k^{Mn} \in Z(G),
$$
where $g_1, \ldots, g_k$ are elements in $H$.

To proceed, let $\tau_n(x_1, x_2, \ldots, x_k) = \tau_n(\overline{x})$ be the $n$th \emph{Petresco word} \cite[p. 40]{MR0283083}, which is defined by the recursive formula,
$$
x_1^n x_2^n \cdots x_k^n = \tau_1(\overline{x})^n \tau_2(\overline{x})^{{n \choose 2}} \cdots \tau_n(\overline{x})^{n \choose n-1}.
$$
By the Hall-Petresco Theorem \cite[p. 41, Theorem 6.3]{MR0283083}, we
have that $\tau_n(H) \subset \gamma_{n}(H)$ for all $n \in \N$.
Thus, replacing $n$ with $Mn$ and using the Hall-Petresco Theorem, we get:
$$
g_1^{Mn} g_2^{Mn} \cdots g_k^{Mn} = \tau_1(\overline{g})^{Mn} \tau_2(\overline{g})^{{{Mn} \choose 2}} \cdots \tau_d(\overline{g})^{{Mn} \choose d}.
$$
By the Hall-Petresco Theorem, $\tau_k(\overline{g}) \in
\zeta_{d-1}(G)$ for all $k>1$, and by definition $h = g_1^{Mn} g_2^{Mn} \cdots g_k^{Mn}
\in \zeta_{d-1}(G)$. Therefore, because $G/\zeta_{d-1}(G)$ is torsion-free, $\tau_1(\overline{g})$ is in $\zeta_{d-1}(G).$
We conclude that, for each $1 \leq k \leq d$, there exists $z_k \in \zeta_{d-1}(G)$ such that
$$
\tau_k(\overline{g})^{Mn \choose k} = (z_k)^{\frac{M}{(d)!}n} \in \zeta_{d-1}(G)^{\frac{M}{(d)!}}.
$$
Further,
$$
\frac{M}{d!} = (d-1)! \cdots 2!.
$$
Hence, $h \in (\zeta_{d-1}(G))^{(d-1)! (d-2)! \cdots 2!}$, so by the inductive hypothesis, we must have
$h \in Z(G)^{n}$, which completes the proof of equation \ref{eqn:claim}.

Let $D$ be the product of all finite order elements in $G/ \gamma_{n}(G)$ for all $n = 1, \ldots, c$.
Selecting $d = c$ in equation \ref{eqn:claim} we get, for $M = (c)! (c-1)! \cdots 2! D$, and $n > c$,
\begin{equation} \label{eqn:tough}
G^{Mn} \cap Z(G) \leq Z(G)^{D n}.
\end{equation}

Now suppose $g \in G^{Mm}$.
By Lemma \ref{lem:technical}, there exists $x_1, \ldots, x_{i-1} \in X$ such that
$[g,x_1, \ldots, x_{i-1}] \in \gamma_t(G) \cap Z(G)$.
Thus, as $G^{Mm}$ is normal, we have, by equation \ref{eqn:tough},
$[g,x_1, \ldots, x_{i-1}] \in Z(G)^{D m}$.
Hence, by our choice of $D$ and the fact that $Z(G)$ is a free abelian group, we have $\nu_G([g,x_1, \ldots, x_{i-1}]) = t$.
Thus, applying Lemma \ref{lem:distortion1} gives $C_1 > 0$, depending only on $G$, such that
$$
\| [g, x_1, \ldots, x_{i-1}] \|_X > C_1 m^{1/t}.
$$
A simple counting argument gives a $C_2 > 0$, depending only on $G$, such that
$$
\| g \|_X \geq C_2 \| [g, x_1, \ldots, x_{i-1}] \|_X.
$$
Thus, we have $C > 0$, depending only on $G$, such that
$$
\| g \|_X > C m^{1/t},
$$
as desired.
\end{proof}

\section{Some examples and basic results} \label{sec:examples}

\subsection{Abelian groups} \label{sec:abeliangroups}

In this section we discuss some facts concerning abelian groups and present a proof of Theorem \ref{thm:abelian}.
This begins with the simplest torsion-free group.
Fix $\{ 1 \}$ as the generating set $\Z$.
Then $B_\Z(n) = \{ -n, -n+1, \ldots , n-1, n \}$.
Clearly, $B_\Z(n)$ is fully detected by $\Z/(2n+1) \Z$.
Further, any quotient fully detecting $B_\Z(n)$ has cardinality greater than $2n$.
So we get $\G_\Z(n) \approx n$.
This result generalizes immediately to all torsion-free finitely
generated abelian groups, and more generally to all finitely generated abelian groups.

\begin{corollary} \label{prop:abelian}
Let $A$ be a finitely generated abelian group.
Then $\G_A(n) \approx n^{\dim(A)}$.
\end{corollary}

\begin{proof}
By Corollary \ref{cor:tfreduction}, we may assume $A$ is torsion-free.
The computation in this case is straightforward.
\end{proof}

One salient consequence of Corollary \ref{prop:abelian} is that an
abelian group's full residual finiteness growth $\G_A$ matches its
word growth $w_A$. We now prove Theorem~\ref{thm:abelian}, which shows
that this property characterizes abelian groups in the class of
nilpotent groups.  It also demonstrates that although $\G_G$ and $w_G$
share properties, they are seldom the same.

\begin{proof}[Proof of Theorem \ref{thm:abelian}]
By Corollary \ref{cor:tfreduction}, we may assume that $G$ is torsion-free.
Let's further assume $G$ is not abelian.
Fix a Malcev basis $x_1, \ldots, x_k$ for $G$.
For every $n$, let $Q_n$ be a quotient fully detecting $B_G(n)$.
Let $c$ be the nilpotent class of $G$.
Fix a tuple $(x_1, \ldots, x_m)$ consisting of all the basis elements not in $\zeta_{c-1}(G)$.
We claim that there exists $C >0$ such that for any $\gamma \in \zeta_{c-1}(G)$, the image of
$x_1^{k_1} x_2^{k_2} \cdots x_m^{k_m} \gamma$
in $Q_{Cn}$ is nontrivial in $Q_{Cn}$ for any $|k_i| \leq n^2$ with $\sum_{i=1}^m |k_i| > 0$.
Indeed, by the second assertion of Lemma \ref{lem:technical} there exists $C > 0$ with
$[x_1^{k_1} x_2^{k_2} \cdots x_m^{k_m}, y_1, y_2, \ldots, y_{c-1}]$
being nontrivial and having word-length at most $Cn^{2/{c}} \leq Cn$ in $G$.
Thus, as nontrivial elements of $B_G(Cn)$ are nontrivial in $Q_{Cn}$,
the image of $x_1^{k_1} x_2^{k_2} \cdots x_m^{k_m} \gamma$ in $Q_{Cn}$
is nontrivial, so the claim is shown.

Consider the set 
$$
B^+(n) := \{ x_1^{k_1} x_2^{k_2} \cdots x_m^{k_m} \gamma : 1 \leq k_i \leq n^2, \gamma \in B_G(n) \cap \zeta_{c-1}(G) \}.
$$
Given any $x,y\in B^+(n)$, the above claim implies that $y^{-1}x$ has
nontrivial image in $Q_{Cn}$. It follows that $B^+(n)$ is fully detected by $Q_{Cn}$.
On the other hand, by comparing with the explicit calculations for
word growth in \cite{Bass72} and the appendix of \cite{MR623534} we see that the set $B^+(n)$ has cardinality at least $n^m w_G(n)$.
Thus, we have
$$
\G_G(n) \succeq n^m w_G(n),
$$
as desired.

\end{proof}

\subsection{Some non-abelian groups} \label{sec:heisenberg}

We begin this section with the simplest non-abelian example.
Recall that the \emph{discrete Heisenberg group} is given by
$$H_3 = \left< x, y, z : [x,y] = z , z \text{ is central } \right>.$$

\begin{proposition} \label{prop:heisenberg}
We have $\G_{H_3}(n) \approx n^6$.
\end{proposition}

\begin{proof}
Let $B_{H_3}(n)$ be the ball of radius $n$ in $H_3$ with respect to
the generating set $\{x,y,z \}$.
An exercise in the geometry of $H_3$ show that there is some $D>0$ so
that if $x^{\alpha_1} y^{\alpha_2} z^{\alpha_3} \in B_{H_3}(Dn)$ then
$\lvert \alpha_i \rvert \leq
  n$ for $i=1,2$ and $\lvert \alpha_3 \rvert \leq n^2$. 
Therefore there is some $C>0$ so that $B_{H_3}(n)$ injects into the quotient
$H_3 / H_3^{Cn^2}$, and so $\G_{H_3}(n) \preceq n^6$.

Now note that $B_{H_3}(5n)$ contains $z^i$ for $-n^2 \leq i \leq n^2$, as
$$
[x^n,y^{j}] z^{k} = z^{nj + k},
$$
has word length at most $5n$ for each $1 \leq j, k \leq n$.
Let $Q_n$ be a quotient detecting $B_{H_3}(5n)$.
Consider $w = x^a y^b x^c$.
Then $[w,y] = z^{a}$ and $[w,x] = z^{-b}$. 
If $w$ is trivial in $Q_n$ then both $[w,y]$ and $[w,x]$ are also trivial.
It follows that $w$ has nontrivial image in $Q_n$ for any values $0 <
a,b,c \leq n^2$.
Thus, $|Q_n| \geq n^6$, as desired.
\end{proof}

We now prove Theorem \ref{theorem:zl} from the introduction, which
generalizes the conclusion of Proposition \ref{prop:heisenberg} to a
large class of nilpotent groups.

For a finite $k$-tuple of elements $X = (x_1, \ldots, x_k)$ from a group, we will use $B_X^+(n)$ to denote the set
$$
B_X^+(n) = \{ x_1^{\alpha_1} \cdots x_k^{\alpha_k} : 0 \leq \alpha_i \leq n \} \subseteq G.
$$
Note that this is \emph{not} generally the same as the semigroup ball
of radius $n$.

\begin{proof} [Proof of Theorem \ref{theorem:zl}]
  Lemma \ref{lem:torsionlengths} demonstrates that $B_G(n)$ is fully detected by a quotient of
  the form $G / G^{Mn^c}$ for some $M>0$. We therefore have 
\[\G_G(n) \preceq \left(\prod_{i=1}^c n^{\dim(\zeta_i(G)/\zeta_{i-1}(G))}\right)^c = n^{\dim(G) c}.\]

To show the reverse inequality, we will show that for any positive integer $n$, there exists set of cardinality approximately $n^{\dim(G)c}$ that is fully detected by any finite quotient of $G$ that realizes $\G_G(n)$.  To this end, for each $i$, equip $\gamma_i(G)$ with a fixed generating set $X_i$.
 Let $Q$ be a quotient of $G$ that realizes $\G_G(n)$.  By \cite[Theorem 2.2]{MR1872804}, for any generating
  set of $\gamma_c(G)$ there is a constant $C> 0$ such that for every
  $h, h'$ in $\gamma_c(G)$, we have
$$
d_{\gamma_c(G)} (h,h') \leq C [d_G(h,h')]^{c}.
$$
Thus, the set $B_{\gamma_c(G)}(n^c /C)$ must inject into $Q$
as it is contained in $B_G(n)$.

To continue, fix a basis $B = (g_1, \ldots, g_k)$ obtained from the upper central series. 
For any $i$, let 
\[
B_i = \left\{ g_{j}\in B \mid g_j \in \zeta_i(G) \setminus
  \zeta_{i+1}(G), g_j \text{ nontorsion in } G/\zeta_i(G) \right\}.
\] 
Set $B^t$ to be the tuple consisting of elements from $B_i$ respecting the ordering of the basis.
That is, 
$$
B^t = (g_{a_1}, \ldots, g_{a_k}),
$$
where each entry is in some $B_i$ and $a_i < a_{i+1}$.
We claim that $B_{B^t}^+(D n^c)$ is fully detected by $Q$ for some $D > 0$.
To prove this claim, we will use the fact that if any element in the
normal closure of some $g\in G$ has nontrivial image in $Q$, then $g$
has nontrivial image in $Q$.
Let $x, y \in B_{B^t}^+(n^c)$ be elements with $x \neq y$.
There is some $i \leq c$ so that $y^{-1}x \in \zeta_i(G) \setminus 
\zeta_{i-1}(G)$. 
Set 
\[
E = \max \{ |B_j| \} \cdot \max \{ \| \gamma \|_{X_j} : \gamma \in
B_j \}.
\]
There is some $\gamma \in \zeta_{i-1}(G)$ so that $\| y^{-1}x \gamma
\|_{X_i} \leq E n^c$. 
This statement follows by reducing the word $y^{-1} x$ to normal form with respect to the basis.
Let $E_0$ be the largest constant $C_i$ output by Lemma
\ref{lem:technical} for $i = 1,\dotsc, c$.
By Lemma \ref{lem:technical},
$$
\| [y^{-1}x , x_1, \ldots, x_r ] \|_{X_c} \leq E_0 \| x y^{-1} \gamma \|_{X_i} \leq E_0 E n^c.
$$
It follows then that the set $B_{B^t}^+(n^c/ (C E_0 E) )$ is fully detected by $Q$.
Set $D = 1/(C E_0 E)$.
By the definition of a basis we have $|B_{B^t}^+( D n^c)| \geq (D n)^{c \dim(G)}$, so we get the desired inequality.
\end{proof}

\section{A general upper bound}
\label{MainProofSection}

The example $H_3 \times \Z$, which has full residual finiteness growth $n^7$, demonstrates that the conclusion of Theorem \ref{theorem:zl} does  not hold for any finitely generated nilpotent group.
In this section we prove Theorem \ref{MainTheorem}, providing a
technique that provides for {\em any} finitely generated nilpotent group an
explicit upper bound of $\G_G(n)$ of the form $n^d$. We first
illustrate the technique in an example in Proposition
\ref{D22example}, where we show moreover that the upper bound is sharp
in this example.

Let $U_n$ denote the group of upper triangular unipotent matrices in
$\SL_n(\Z)$. 
For $i\neq j$, let $e_{i,j}$ denote the elementary matrix differing
from the identity matrix only in that its $ij$-entry is 1.
We define the \emph{coordinates} of the tuple $(x_1, \ldots, x_k)$ to be the set
$\{ x_1, \ldots, x_k \}$.
Recall that a {\em terraced} filtration of $G$ is a filtration $1 = H_0 \leq
H_1 \leq \dotsb \leq H_{c-1}\leq G$ where each $H_i$ is a maximal normal
subgroup of $G$ satisfying $H_i\cap \gamma_{i+1}(G) = 1$.

\begin{proposition} \label{D22example}
  Consider elementary matrices $x = e_{1,4}$ and $y = e_{1,5}$ in
  $U_5$. Define a normal subgroup $N = \left< x, y \right> \leq
  U_5$ and set $\Gamma = U_5 / N$. Then $\G_{\Gamma}(n) \approx n^{22}.$
\end{proposition}

\begin{proof}
  Set $H_3 = \Gamma$ and $H_2 = \left< e_{1,2}, e_{1,3} \right>$, and
  let $H_0 = H_1 = 1$. 
  Note that $1=H_0\leq H_1 \leq H_2 \leq \Gamma$
  forms a terraced filtration of $\Gamma$.
  Define two tuples of elements of $\Gamma$ by $X_3 = (e_{1,3} ,
  e_{1,2})$ and $X_2 = (e_{2,5}, e_{2,4}, e_{3,5}, e_{2,3}, e_{3,4},
  e_{4,5})$. For each $i=2,3$, let $Y_i$ be the set of coordinates of
  $X_i$. Clearly $Y = Y_2 \cup Y_3$ generates $\Gamma$. 

  To establish the upper bound, let $Q$ be a quotient of $\Gamma$ detecting $B_\Gamma(n)$.
  Each
  of $H_3^{n^3}$ and $H_2^{n^2}$ is normal in $\Gamma$, so we can define a normal
  subgroup $N = H_3^{n^3} H_2^{n^2} \leq \Gamma$. 
  A simple induction shows that if $g\in B_\Gamma(n)$ then $\lvert
  g_{ij} \rvert \leq n^{j-i}$. In particular this implies that there
  is some $C>0$ so that $B_\Gamma(Cn)$ is fully detected by
  $G/N$. Since $\lvert G/N \rvert \approx n^{22}$, this establishes
  the desired upper bound on $\G_\Gamma(n)$.
  
  To establish the lower bound, define the \emph{depth} of an element
  $\gamma \in \Gamma$ to be the maximal $i$ with
  $\gamma \notin \zeta_i(\Gamma)$.  Order the elements $Y$ in a tuple
  $(y_1, y_2, \ldots, y_8)$ of non-increasing depth.  Set $B^+(n)$ to
  be
  $$
  \left\{ \prod_{i=1}^8 y_i^{\alpha_i}  : 0 \leq \alpha_i \leq n^2 \text{ if $y_i \in Y_2$ and } 0 \leq \alpha_i \leq n^3 \text{ otherwise} \right\}.
  $$
  We claim that there exists $C >0$ such that any quotient $Q$ in which $B_\Gamma(Cn)$ embeds restricts to $B^+(n)$ as an injection.
  This gives the desired lower bound, as $|B^+(n)| \geq n^{22}$.
  To see this claim, let $x,y$ be distinct elements in $B^+(n)$.
  Set $i$ to be the depth of $y^{-1}x$.
  We break up the rest of the proof of this claim into cases depending on $i$.
  
  If $i = 0$, then
  $y^{-1} x$ is in the center of $\Gamma$ and we have
  $$
  y^{-1} x = e_{1,2}^{a_1} e_{2,5}^{a_2},
  $$
  where $|a_1| \leq n^2$ and $|a_2| \leq n^3$.
  Note that $e_{1,2}\in \gamma_2(\Gamma)$ and $e_{2,5}\in \gamma_3(\Gamma)$.
  Applying Lemma \ref{lem:distortion} twice, we have that
  $$
  \| y^{-1} x\|_\Gamma \leq \| e_{1,2}^{a_1} \|_\Gamma + \|
  e_{2,5}^{a_2} \|_\Gamma \leq C n,
  $$
  for some $C > 0$, independent of $n$.
  Thus $y^{-1} x$ cannot vanish in any quotient that fully detects $B^+(Cn)$.
  
  If $i = 1$, then by definition, we may write
  $$
  y^{-1} x = e_{1,2}^{a_1} e_{2,4}^{a_2} e_{3,5}^{a_3} \gamma,
  $$
  where $\gamma \in \zeta_i(\Gamma)$, $|a_1| \leq n^2$, $|a_2| \leq n^3$, and $|a_3| \leq n^3$.
  Since this $y^{-1}x$ is not in the center, there exists $z \in Y$ such that 
  $$
  [e_{1,2}^{a_1} e_{2,4}^{a_2} e_{3,5}^{a_3} \gamma, z] \neq 1.
  $$
  This element is now in the center. Thus, by Lemma \ref{lem:commutatorproduct}, we have
  $$
  [e_{1,2}^{a_1} e_{2,4}^{a_2} e_{3,5}^{a_3} \gamma, z]
  = 
  [e_{1,2}, z]^{a_1}
  [e_{2,4}, z]^{a_2}
  [e_{3,5}, z]^{a_3}.
  $$
  Now by Lemma \ref{lem:distortion} applied three times, we see that the word length of
  $[e_{1,2}^{a_1} e_{2,4}^{a_2} e_{3,5}^{a_3} \gamma, z]$
  is less than a constant multiple of $n$, where the constant does not depend on $n$.
  Thus $y^{-1} x$ cannot vanish in any quotient that fully detects $B^+(Cn)$ for some  $C > 0$ independent of $n$.
  
  If $i = 2$, then by definition, we may write
  $$
  y^{-1} x = e_{2,3}^{a_1} e_{3,4}^{a_2} \gamma,
  $$
  where $\gamma \in \zeta_i(\Gamma)$, $|a_1|, |a_2| \leq n^3$.
  Suppose, without loss of generality, that $a_1 \neq 0$.
  Then, using Lemma \ref{lem:commutatorproduct}, we have that there exists $\gamma' \in \gamma_1(\Gamma)$ such that
  $$
  [y^{-1} x, e_{3,4}] = [e_{2,3}, e_{3,4}]^{a_1} [e_{3,4}, e_{3,4}]^{a_2} \gamma' = e_{2,4}^{a_1} \gamma'.
  $$
  Now it is clear that there exists $z \in \Gamma$ such that
  $$
  [[y^{-1} x, e_{3,r}], z] = [e_{2,4}^{a_1} \gamma', z] \neq 1.
  $$
  We can proceed as in case $i=1$ to achieve the desired conclusion.
  Indeed, Lemma \ref{lem:distortion} applies, giving that $y^{-1} x$ is detected if $B(Cn)$ is fully detected for some constant $C > 0$ independent of $n$.
  That is, we cannot have $y^{-1}x =1$ in $Q$, if $Q$ detects $B_\Gamma(Cn)$.
  The claim then follows, ending the proof.
  
\end{proof}

We now prove Theorem \ref{MainTheorem}.

\begin{proof}[Proof of Theorem \ref{MainTheorem}]
Let $G$ be a finitely generated nilpotent group and suppose $1=H_0
\leq H_1 \leq \dotsb \leq H_{c-1}\leq G$ is a terraced filtration. Set
$H_c = G$. 

Choose a basis $X_1$ of $H_1$. Inductively construct
tuples $X_2,\dotsc, X_c$ by setting $X_i$ to be a pull-back of a basis
for $H_i / H_{i-1}$.
Set $Y_i$ to be the set of all coordinates of $X_i$ and $Y = \cup_i Y_i$.
It is clear from the construction that $Y$ is generating set for
$G$. Note also that for any $n \in \N$, the subgroup 
\[
N(n) = \prod_{i=1}^c \left<
  y^{n^k} : y \in \left<  Y_1 \cup Y_2 \cup \dotsb \cup Y_k
  \right> \right>
\]
is normal in $G$.

We now claim that there exists a constant $D \in \N$ so that for any
$n \in \N$, the ball $B_Y(n)$ is detected by $G/N(Dn)$.
To prove the claim, let $f, M \in \N$ be as in Lemma~\ref{lem:torsionlengths}.
Then $G^{fM}$ is torsion-free; let $K = G^{f M}$.
Fix a finite generating set $T$ for $K$.
For each $i$ and any $n\in \N$, Lemma \ref{lem:torsionlengths} gives
that any element $g\in K^n \cap H_i$ has word length at least $C_i
n^{1/t_i}$ with respect to $T$.
Thus we have that  there exists $D_0 > 0$ such that
$B_T(D_0 n)$ is fully detected by $K/N(fMn)$.
Further, since $K$ is of finite index in $G$, we have $D_1 > 1$ such that for any $g \in K$,
$$
\| g \|_T \leq D_1 \| g \|_Y \leq D_1^2  \| g \|_T.
$$
Therefore, as $N(fMn)$ is contained in $K$, any singleton contained in $B_Y(n/D_1)$ is fully detected by $G/N(fMn)$ and so $B_Y(n/(2 D_1))$ is fully detected by $G/N(fMn)$.
This proves the claim, as we can select $D = 2D_1 fM$.

We will now demonstrate that the order of
$G/N(Dn)$ is dictated by a single polynomial of the form $n^b$ for 
\[
b = \sum_{k=1}^c k\cdot \dim(H_k / H_{k-1}).
\]
Set $G_k = H_k/ H_{k-1}$. It is apparent from the definition of $N(Dn)$ the index of $N(Dn)$ in $G$ is bounded above by
$$
\prod_{k=1}^{c} | G_k / G_k^{D^k n^{k}} |.
$$
By the construction of $D$, the subgroup $G_k^D$ is torsion-free in $G_k$.
Thus, it is clear that $|G_k^D / G_k^{D^{k} n^{k}}|$ has order
$D^{\dim(G_k) (k-1)}  n^{k \dim(G_k)}$.
This gives an upper bound for the index of $N(Dn)$ in $G$ of the form
$C_0 n^{\sum_{k=1}^c k \dim(G_k)}$, where $C_0 >0$ does not depend on $n$.

One can check that $b = c\dim(G) - \sum_{i=1}^{c-1}\dim(H_i)$ using
the general fact that $\dim(G/H) = \dim(G)- \dim(H)$ for any finitely
generated nilpotent group $G$ with normal subgroup $H$. This 
completes the proof since $G/N(Dn)$ detects $B_Y(Cn)$.
\end{proof}

We conclude with an example that shows that the upper bound to $\G_G$
given by Theorem \ref{MainTheorem} generally may depend on choice of
terraced filtration. Consider the group
$\tilde G = U_3 \times U_4 \times U_5$, which is nilpotent of class
$c=4$.  There is an isomorphism
$Z(\tilde G) \cong Z(U_3) \times Z(U_4) \times Z(U_5)$. Under
identifications $Z(U_3)\cong Z(U_4) \cong Z(U_5) \cong \Z$, define an
infinite cyclic subgroup
\[
Z = \{ (x,y,z)\in Z(U_3)\times Z(U_4)\times Z(U_5) \mid x=y=z \} \leq
Z(\tilde G).
\]
Let $G = \tilde G / Z$ and let $\pi: \tilde G \to G$ be the quotient
map. Then $\pi$ restricts to an isomorphism $Z(U_3)\times Z(U_4) \cong
Z(G)$. Under this identification, the last term of the lower central
series of $G$ is
\[
\gamma_4(G) = \{ (x,y)\in Z(U_3)\times Z(U_4) \mid x=y \}.
\]
Since $\gamma_3(G)$ contains the image of $Z(U_4)$, we see that
$Z(G) \leq \gamma_3(G)$. Since $H\cap
Z(G)$ is nontrivial for any nontrivial normal subgroup $H\leq G$, it follows that
$H_2$ is trivial for any terraced filtration of $G$.

Now define $H_0 = H_1 = H_2 = 1$ and $H_3 = \pi(U_3)$, and
$H_0' = H_1' = H_2' = 1$ and $H_3' = \pi(U_4)$. It is easy to see that
both $\pi(U_3)$ and $\pi(U_4)$ are maximal normal subgroups of $G$
whose intersection with $\gamma_4(G)$ is trivial. It follows from the
above comments that
\[
H_0\leq H_1 \leq H_2 \leq H_3 \leq G \quad \text{ and } \quad H_0'
\leq H_1' \leq H_2' \leq H_3' \leq G
\]
are terraced filtrations of $G$. However these filtrations give
different upper bounds for $\G_G$ because $\dim( \pi(U_3) ) = 3$ while
$\dim( \pi(U_4) ) = 6$.

\bibliography{fullrfgrowthnilIJ}
\bibliographystyle{ijmart}


\noindent
Khalid Bou-Rabee \\
Department of Mathematics, CCNY CUNY \\
E-mail: khalid.math@gmail.com \\

\noindent
Daniel Studenmund \\
Department of Mathematics, University of Chicago \\
E-mail: dhs@math.uchicago.edu \\

\end{document}